\documentclass{article}[12pt]
\usepackage[dvips]{graphicx}
\usepackage{amsfonts}
\usepackage{amsmath}
\usepackage{amssymb}
\usepackage{amscd}
\usepackage{color}
\usepackage{amsthm}
\usepackage{indentfirst}
\usepackage[hmargin=3cm,vmargin=3cm]{geometry}
\usepackage{parskip}

\newtheorem{thm}{Theorem}
\newtheorem{thm*}{Theorem}
\newtheorem{prop}{Proposition}[section]
\newtheorem{lma}{Lemma}

\newtheorem{cor}{Corollary}
\newtheorem{clm}{Claim}

\theoremstyle{definition}

\theoremstyle{remark}

\newtheorem{rmk}{Remark}[subsection]

\newcommand{\R}{{\mathbb{R}}}
\newcommand{\Z}{{\mathbb{Z}}}
\newcommand{\C}{{\mathbb{C}}}

\newcommand{\D}{{\mathbb{D}}}

\newcommand{\del}{\partial}

\newcommand{\sm}[1]{C^\infty(#1)}

\newcommand\vol{\operatorname{vol}}
\newcommand\sign{\operatorname{sign}}
\newcommand\lk{\operatorname{lk}}

\newcommand{\G}{\mathcal{G}}

\newcommand{\til}[1]{\widetilde{#1}}

\newcommand{\om}{\omega}

\newcommand{\eps}{\epsilon}

\newcommand{\cM}{\mathcal{M}}

\def\Sign{\operatorname{\textbf{Sign}}}

\DeclareMathOperator{\Diff}{\mathrm{Diff}}
\DeclareMathOperator{\Ham}{\mathrm{Ham}}

\DeclareMathOperator{\Cal}{\mathrm{Cal}}

\DeclareMathOperator{\id}{\mathbf{1}}

\def\o{\omega}

\begin{document}

\title{\textbf{The $L^p$-diameter of the group of area-preserving diffeomorphisms of $S^2$}}

\author{\textsc{Michael Brandenbursky and Egor Shelukhin}\\}

\date{}
\maketitle

\begin{abstract}
We show that for each $p \geq 1,$ the $L^p$-metric on the group of area-preserving diffeomorphisms of the two-sphere has infinite diameter. This solves the last open case of a conjecture of Shnirelman from 1985. Our methods extend to yield stronger results on the large-scale geometry of the corresponding metric space, completing an answer to a question of Kapovich from 2012. Our proof uses configuration spaces of points on the two-sphere, quasimorphisms, optimally chosen braid diagrams, and, as a key element, the cross-ratio map $X_4(\C P^1) \to \mathcal{M}_{0,4} \cong \C P^1 \setminus \{\infty,0,1\}$ from the configuration space of $4$ points on $\C P^1$ to the moduli space of complex rational curves with $4$ marked points.

\end{abstract}


\tableofcontents


\section{Introduction and main results}

\subsection{Introduction}

The $L^2$-length of a path of volume preserving diffeomorphisms, which describes a time-dependent flow of an ideal incompressible fluid, corresponds to the hydrodynamic action of the flow in the same way as the length of a path in a Riemannian manifold corresponds to its energy (cf. \cite{ShnirelmanGeneralized}). Indeed, it is the length of this path with respect to the formal right-invariant Riemannian metric on the group $\G$ of volume preserving diffeomorphisms introduced by Arnol'd in \cite{ArnoldGeodesics}. The $L^1$-length of the same path has a dynamical interpretation as the average length of a trajectory of a point under the flow.

Following the principle of least action, it therefore makes sense to consider the infimum of the lengths of paths connecting two fixed volume preserving diffeomorphisms. This gives rise to a right-invariant distance function (metric) on $\G.$ Taking the identity transformation as the initial point, Arnol'd observes that a path whose $L^2$-length is minimal (and equal to the distance) necessarily solves the Euler equation of an ideal incompressible fluid.

It follows from works of Ebin and Marsden \cite{EbinMarsden} that for diffeomorphisms in $\G$ that are $C^2$-close to the identity, the infimum is indeed achieved. Further, more global results on the corresponding Riemannian exponential map were obtained in \cite{EbinPrestonMisiolek},\cite{ShnirelmanRuled} (see \cite{EbinSymp} for additional references). In \cite{ShnirelmanGeometry,ShnirelmanGeneralized} Shnirel'man showed, among a number of surprising facts related to this subject, that in the case of the ball of dimension $3,$ the diameter of the $L^2$-metric is bounded. This result is known\footnote{The authors have not found a detailed proof of this generalization in the literature.} to hold for all compact simply connected manifolds of dimension $3$ or larger (see \cite{EliashbergRatiu, KhesinWendt, ArnoldKhesin}), while its analogue in the non-simply-connected case is false \cite{EliashbergRatiu,BrandenburskyLpMetrics}. Furthermore, Shnirel'man has conjectured that for compact manifolds of dimension $2,$ the $L^2$-diameter is infinite. 

In this paper we consider Shnirel'man's conjecture, and its analogues for $L^p$-metrics, with $p \geq 1.$ It follows from results of Eliashberg and Ratiu \cite{EliashbergRatiu} that on compact surfaces (possibly with boundary) other than $T^2$ and $S^2,$ Shnirel'man's conjecture holds for all $p \geq 1.$ Their arguments rely on the Calabi homomorphism $\Cal$ \cite{CalabiHomomorphism} from the compactly supported Hamiltonian group $\Ham_c(\Sigma,\sigma)$ to the real numbers in the case of a surface $\Sigma$ non-empty boundary ($\sigma$ is the area form), and on non-trivial first cohomology combined with trivial center of the fundamental group in the closed case. For the two-torus $T^2$ this Shnirel'man's conjecture also holds for all $p\geq 1$, as can be quickly seen by the following steps. First, the methods of \cite{EliashbergRatiu} together with the fact that the Hamiltonian group $\Ham(T^2,dx \wedge dy)$ is simply-connected as a topological space (see e.g. \cite[Chapter 7.2.B]{PolterovichBookGeometry}) imply that $\Ham(T^2,dx \wedge dy)$ with the $L^p$-metric has infinite diameter (compare \cite[Theorem 1.2]{BrandenburskyKedra2}). Second, the inclusion $\Ham(T^2,dx \wedge dy) \hookrightarrow \Diff_0(T^2,dx \wedge dy),$ the two groups being equipped with their respective $L^p$-metrics, is a quasi-isometry (see Proposition \ref{Proposition: torus qi}). The case of the two-sphere $S^2$, to which previous methods do not apply, remained open. 


The case $p>2$ (but not that of Shnirel'man's original conjecture!) is well-known, as it follows from a result of Polterovich  \cite{PolterovichS2} regarding Hofer's metric on $\Ham(S^2)$ by an application of the Sobolev inequality. The authors gave a different proof of this case by elementary methods in the preprint \cite{BrandenburskyESIHES}.

The main result of this paper is the unboundedness of the $L^p$-metric on $\Ham(S^2)$ for all $p \geq 1.$ This completes a full answer to Shnirel'man's question. Our methods extend to yield stonger results on the large-scale geometry on the $L^p$-metric on $\Ham(S^2).$ In particular, we provide bi-Lipschitz group monomorphisms of $\R^m$ endowed with the standard (say Euclidean) metric into $(\Ham(S^2),d_{L^p})$ for each for each positive integer $m$ and each $p \geq 1.$ Moreover, our key technical estimate implies by an argument of Kim-Koberda \cite{KimKoberdaAntiTrees} (cf. Crisp-Wiest \cite{CrispWiest}, Benaim-Gambaudo \cite{BenaimGambaudo}) the existence of quasi-isometric group monomorphisms from each right-angled Artin group to $(\Ham(S^2),d_{L^p})$ for each $p \geq 1$, completing the resolution of a question of Kapovich \cite{KapovichRAAGs} in the case of $S^2$ (the case $p > 2$ shown in Kim-Koberda uses \cite{BrandenburskyESIHES}).

Our methods are two-dimensional in nature, and have to do with braiding and relative rotation numbers of trajectories  of time-dependent two-dimensional Hamiltonian flows (in extended phase space). We note that Shnirel'man has proposed to use relative rotation numbers to bound from below the $L^2$-lengths of two-dimensional Hamiltonian paths in \cite{ShnirelmanGeneralized}. This direction is related to the method of Eliashberg and Ratiu by a theorem of Gambaudo and Ghys \cite{GambaudoGhysEnlacements} and Fathi \cite{FathiThese} (compare \cite{ESER}), stating that the Calabi homomorphism is proportional to the relative rotation number of the trajectories of two distinct points in the two-disc $\D$ under a Hamiltonian flow, averaged over the configuration space of ordered pairs of distinct points $(x_1,x_2)$ in the two-disc. 

This line of research was pursued in \cite{GambaudoLagrange}, \cite{BenaimGambaudo}, \cite{CrispWiest}, \cite{BrandenburskyLpMetrics}, \cite{BrandenburskyKedra1}, \cite{KimKoberdaAntiTrees}, obtaining quasi-isometric and bi-Lipschitz embeddings of various groups (right-angled Artin groups and additive groups of finite-dimensional real vector spaces) into $\Ham_c(\D^2,dx\wedge dy)$ and into $\ker(\Cal) \subset \Ham_c(\D^2,dx\wedge dy)$ endowed with their respective $L^p$-metrics (see \cite{BrandenburskyKedra2} for similar embedding results on manifolds with sufficiently complicated fundamental group). In all cases, the key technical estimate is an upper bound, via the $L^p$-length of an isotopy of volume-preserving diffeomorphisms, of the average, over all points in the manifold, of the word length in the fundamental group of the trace of the point under the isotopy (closed up to a loop by a system of short paths on the manifold). In this paper we produce similar estimates for the case of the two-sphere. Our case of $\Diff_0(S^2,\sigma),$ with $p \leq 2,$ is more difficult than that of ${\ker(\Cal) \subset \Ham_c(\D,dx\wedge dy)}$ because the required analytical and topological bounds require a more global approach and have to take into account the geometry and topology of the sphere.


In turn, lower bounds on the average word length can often be provided by quasimorphisms - functions that are additive with respect to the group multiplication - up to an error which is uniformly bounded (as a function of two variables). The quasimorphisms we use were introduced and studied by Gambaudo and Ghys in the beautiful paper \cite{GambaudoGhysCommutators} (see also \cite{PolterovichDynamicsGroups},\cite{PyTorus},\cite{PySurfacesQm},\cite{BrandenburskyKnots}). These quasimorphisms essentially appear from invariants of braids traced out by the action of a Hamiltonian path on an ordered $n$-tuple of distinct points in the surface (suitably closed up), averaged over the configuration space $X_n(\Sigma)$ of $n$-tuples of distinct points on the surface $\Sigma$. 

\textbf{Comparison with \cite{BrandenburskyESIHES}}

The first step in our study of Shnirel'man's conjecture is found in the unpublished preprint \cite{BrandenburskyESIHES}, where we saw which elements of the approach of \cite{BrandenburskyLpMetrics} extend to the case of $S^2.$ There we found that without a key new idea one could only obtain the necessary estimates for $p>2.$ The main novelty of this paper consists indeed of a new geometric idea, which is of independent interest. To wit, we introduce certain canonical "logarithmic" differential forms on $X_n(\C P^1),$ which play a key role in our arguments. These forms can be considered analogues for the case of $\C P^1$ of the differential forms of Arnol'd \cite{ArnoldColoredBraids} on $X_n(\C)$. One curious aspect of these forms is that while in Arnol'd's case they appeared from pairs of points, that is from the natural projections $X_n(\C) \to X_2(\C)$ on pairs of coordinates, in the case of $\C P^1$ they are constructed from quadruples of points, that is from projections $X_n(\C P^1) \to X_4(\C P^1)$ on quadruples of coordinates. This fits with $P_2(\C) = \pi_1(X_2(\C)) \cong \pi_1(\C \setminus \{0\}) = \Z$ and $P_4(\C P^1) = \pi_1(X_4(\C P^1)) \cong \Z/2\Z \times \pi_1(\C \setminus \{0,1\}) = \Z/2\Z \times (\Z * \Z)$ being the first infinite pure braid groups in the two cases.

\subsection{Preliminaries}

\subsubsection{The $L^p$-metric}

Let $M$ denote a smooth oriented manifold without boundary that is either closed, or $M = X \setminus \del X$ for a compact manifold $X.$ Let $M$ be endowed with a Riemannian metric $g$ and smooth measure $\mu$ (given by a volume form, which in our case that $M$ is a surface is an area form $\sigma,$ and orientation on $M$). We require $g$ and $\mu$ to extend continuously to $X$ in the second case. Finally denote by \[\G=\Diff_{c,0}(M,\mu)\] the identity component of the group of compactly supported diffeomorphisms of $M$ preserving the smooth measure $\mu.$

Fix $p\geq 1.$ For a smooth isotopy $\{\phi_t\}_{t \in [0,1]},$ from $\phi_0 = \mathrm{1}$ to $\phi_1 = \phi,$  we define the $L^p$-length by \[{l}_p(\{\phi_t\}) = \int_0^1  \left(\frac{1}{\vol(M,\mu)} \cdot \int_M |X_t|^p d\mu\right)^{\frac{1}{p}} \, dt \,,\] where $X_t = \frac{d}{dt'}|_{t'=t} \phi_{t'} \circ \phi_t^{-1}$ is the time-dependent vector field generating the isotopy $\{\phi_t\}$, and $|X_t|$ is its length with respect to the Riemannian structure on $M$. As is easily seen by a displacement argument, the $L^p$-length functional determines a non-degenerate norm on $\G$ by the formula \[d_p(\id,\phi) = \inf \; l_p(\{\phi_t \}).\] This in turn defines a right-invariant metric on $\G$ by the formula \[d_p(\phi_0,\phi_1) = d_p(\id,\phi_1 {\phi_0}^{-1}).\]

\begin{rmk}
Consider the case $p = 1.$ It is easy to see that the $L^1$-length of an isotopy is equal to the average Riemannian length of the trajectory $\{\phi_t (x)\}_{t \in [0,1]}$ (over $x \in M,$ with respect to $\mu$). Moreover for each $p \geq 1,$ by Jensen's (or H\"{o}lder's) inequality, we have \[l_p(\{\phi_t\}) \geq l_1(\{\phi_t\}).\]
\end{rmk}

Denote by $\til{\id}$ the identity element of the universal cover $\til{\G}$ of $\G.$ Similarly one has the $L^p$-pseudo-norm (that induces the right-invariant $L^p$-pseudo-metric) on $\til{\G}$, defined for $\til{\phi} \in \til{\G}$ as
\[d_p(\til{\id},\til{\phi}) = \inf \; l_p(\{\phi_t\}),\] where the infimum is taken over all paths $\{\phi_t\}$ in the class of $\til{\phi}.$ Clearly $d_p(\id,\phi) = \inf d_p(\til{\id},\til{\phi}),$ where the infimum runs over all $\til{\phi} \in \til{\G}$ that map to $\phi$ under the natural epimorphism $\til{\G} \to \G.$

Up to bi-Lipschitz equivalence of metrics ($d$ and $d'$ are equivalent if $\frac{1}{C}d \leq d' \leq C d$ for a certain constant $C>0$) the $L^p$-metric on $\G$ (and its pseudo-metric analogue on $\til{\G}$) is independent of the choice Riemannian structure and of the volume form $\mu$ on $M.$ In particular, the question of boundedness or unboundedness of the $L^p$-metric enjoys the same invariance property.

{\bf Terminology:} For a positive integer $n,$ we use $A,B,C > 0$ as generic notation for positive constants that depend only on $M,\mu,g$ and $n.$

\subsubsection{Quasimorphisms}

For some of our results, we require the notion of a quasimorphism. Quasimorphisms are a helpful tool for the study of non-abelian groups, especially those that admit few homomorphisms to $\R.$ A quasimorphism $r: G \to \R$ on a group $G$ is a real-valued function that satisfies
\[r(xy) = r(x) + r(y) + b_r(x,y),\]
for a function $b_r:G\times G \to \R$ that is uniformly bounded:
\[\delta(r): = \sup_{G\times G} |b_r| < \infty.\]

A quasimorphism $\overline{r}:G \to \R$ is called \textit{homogeneous} if $\overline{r}(x^k) = k \overline{r}(x)$ for all $x\in G$ and $k \in \Z$. In this case, it is additive on each pair $x,y \in G$ of commuting elements: $r(xy) = r(x) + r(y)$ if $xy = yx.$

For each quasimorphism $r:G\to \R$ there exists a unique homogeneous quasimorphism $\overline{r}$ that differs from $r$ by a bounded function: \[\sup_G |\overline{r} - r| < \infty.\]
It is called the \textit{homogenization} of $r$ and satisfies
\[\overline{r}(x) = \lim_{n \to \infty} \frac{r(x^n)}{n}.\]

Denote by $Q(G)$ the real vector space of homogeneous quasimorphisms on $G.$

For a finitely-generated group $G,$ with finite symmetric generating set $S,$ define the word norm $|\cdot |_S:G \to \Z_{\geq 0}$ by \[|g|_S = \min \{k\,|\, g = s_1 \cdot \ldots \cdot s_k,\,  \forall\, 1 \leq j \leq k, \, s_j \in S\}\] for $g \in G.$ This is a norm on $G,$ and as such it induces a right-invariant metric $d_S: G \times G \to \Z_{\geq 0}$ by $d_S(f,g) = |gf^{-1}|_S.$ This metric is called the word metric. In this setting, any quasimorphism $r :G \to \R$ is controlled by the word norm. Indeed, for all $g \in G,$ \[|r(g)| \leq \left(\delta(r) + \max_{s \in S} |r(s)|\right) \cdot |g|_S.\]

We refer to \cite{CalegariScl} for more information about quasimorphisms.

\subsubsection{Configuration spaces and braid groups}
For a manifold $M,$ which shall in this paper be usually of dimension $2$ and without boundary, the configuration space $X_n(M) \subset M^n$ of $n$-tuples of points on $M$ is defined as \[X_n(M) = \{(x_1,\ldots,x_n)|\; \displaystyle{x_i \neq x_j,}\;  1 \leq i < j \leq n\}.\] That is \[X_n(M) = M^n \setminus \bigcup_{1 \leq i < j \leq n} D_{ij}\] where for $1 \leq i < j \leq n,$ the partial diagonal $D_{ij} \subset M^n$ is defined as $D_{ij} = \{(x_1,\ldots,x_n)|\; x_i = x_j\}.$ Note that $D_{ij}$ is a submanifold of $M^n$ of codimension $\dim M.$ When $\dim M = 2$ and $M$ is endowed with a complex structure, $D_{ij}$ is a complex hypersurface. Therefore we shall sometimes refer to $D$ as a divisor. Indeed, complex coordinates serve an important role in our arguments.

Finally we define the pure braid group of $M$ as \[P_n(M) = \pi_1(X_n(M)).\] Noting that the symmetric group $S_n$ on $n$ elements acts on $X_n(M),$ we form the quotient $C_n(M) = X_n(M)/S_n$ and define the full braid group of $M$ as \[B_n(M) = \pi_1(C_n(M)).\] For smooth surfaces $M$ endowed with a complex structure (hence smooth complex manifolds of complex dimension $1$), $C_n(M)$ turn out to inherit the structure of a smooth complex manifold of complex dimension $n.$

We note that $P_n(M)$ and $B_n(M)$ enter the exact sequence $1 \to P_n(M) \to B_n(M) \to S_n \to 1.$ In particular $P_n(M)$ is a normal sugbroup of $B_n(M)$ of finite index. We refer to \cite{TuraevKassel} for further information about braid groups.

\subsubsection{Short paths and the Gambaudo-Ghys construction}

Given a real valued quasimorphism $r$ on $P_n(M) = \pi_1(X_n(M),q)$ for a fixed basepoint $q \in X_n(M)$ there is a natural way to construct a real valued quasimorphism on the universal cover $\widetilde{\G}$ of the group $\G=\Diff_0(S^2,\sigma)$ of area preserving diffeomoprhisms of $M=S^2$. We shall see that in our case of $M=S^2$ this induces a quasimorphism on $\G$ itself, because the fundamental group of $\G$ is finite. The construction is carried out by the following steps (cf. \cite{GambaudoGhysCommutators,PolterovichDynamicsGroups,BrandenburskyKnots}).

\begin{enumerate}
\item For all $x \in X_n(S^2) \setminus Z$, with $Z$ a closed negligible subset (e.g. a union of submanifolds of positive codimension) choose a smooth path $\gamma(x):[0,1] \to X_n(S^2)$ between the basepoint $q\in X_n(S^2)$ and $x$. Make this choice continuous in $ X_n(S^2) \setminus Z$. We first choose a system of paths on $M=S^2$ itself, in our case the minimal geodesics with respect to the round metric, and then consider the induced coordinate-wise paths in $M^n$, and pick $Z$ to ensure that these induced paths actually lie in $X_n(S^2)$. After choosing the system of paths $\{\gamma(x)\}_{x \in X_n(S^2)\setminus Z}$ we extend it measurably to $X_n(S^2)$ (obviously, no numerical values computed in the paper will depend on this extension). We call the resulting choice a "system of short paths". 


\item Given a path $\{\phi_t \}_{t \in [0,1]}$ in $\G$ starting at $Id$, and a point $x \in X_n(S^2)$ consider the path $\{\phi_t \cdot x\}$, to which we then catenate the corresponding short paths. That is consider the loop \[\lambda(x,\{\phi_t\}) := \gamma(x) \# \{\phi_t \cdot x\} \# \gamma(y)^{-1}\] in $X_n(S^2)$ based at $q,$ where $^{-1}$ denotes time reversal. Hence we obtain for each $x \in X_n(S^2)$ an element $[\lambda(x,\{\phi_t\})] \in \pi_1(X_n(S^2),q)$.

\item Consequently applying the quasimorphism $r:\pi_1(X_n(S^2),q) \to \R$ we obtain a measurable function $f:X_n(S^2) \to \R$. 	   Namely $f(x) = r([\lambda(x,\{\phi_t\})])$. The quasimorphism $\Phi$ on $\widetilde{\G}$ is defined by    \[\Phi([\{\phi_t\}]) = \int_{X_n(S^2)} f \,d \mu^{\otimes n}.\]
     It is immediate to see that this function is well-defined by topological reasons. The quasimorphism property follows by the quasimorphism property of $r$ combined with finiteness of volume. The fact that the function $f$ is absolutely integrable can be shown to hold a-priori by a reduction to the case of the disc. We note, however, that by Tonelli's theorem this fact follows as a by-product of the proof of our main theorem, and therefore requires no additional proof.
\item Of course our quasimorphism can be homogenized, to obtain a homogeneous quasimorphism $\overline{\Phi}$.
\end{enumerate}
\vspace{2mm}
\begin{rmk}
In our case, by the result of Smale \cite{Smale} $\pi_1(\G) = \Z/2\Z$, and hence the quasimorphisms descend to quasimorphisms on $\G$, e.g. by minimizing over the two-element fibers of the projection $\widetilde{\G} \to \G$. For $\overline{\Phi}$, the situation is easier since by homogeneity it vanishes on
$\pi_1 (\G) \subset Z(\til{\G})$, and therefore depends only on the image in $\G$ of an element in $\til{\G}$. We keep the same notations for the induced quasimorphisms.
\end{rmk}

\subsubsection{The cross-ratio map}
Recall that $S^2$ can be identified with $\C P^1,$ and the latter has an affine chart $u_0: \C \to \C P^1,$ $u_0(z) = [z,1],$ in homogeneous coordinates, whose image is the complement of the point $\infty:=[1,0]$.   

The cross-ratio map is given by the natural\footnote{Recall that the holomorphic automorphism group of $\C P^1$ is isomorphic to $\mathrm{PSL}(2,\C)$ acting by fractional-linear transformations.} projection $X_4(\C P^1) \to \mathcal{M}_{0,4} = X_4(\C P^1)/ \mathrm{PSL}(2,\C).$ Composing it with the isomorphism $\cM_{0,4} \cong \C P^1 \setminus \{\infty,0,1\} \cong \C \setminus \{0,1\}$ given by the inverse of the map $u \mapsto [(\infty,0,1,u)],$ we obtain a map \[cr:X_4(\C P^1) \to \C \setminus \{0,1\}.\] In other words $cr(x_1,x_2,x_3,x_4) = A(x_4)$ for the unique map $A \in \mathrm{PSL}(2,\C)$ with $A(x_1) = \infty,\; A(x_2) = 0,\; A(x_3) = 1.$

In homogeneous coordinates, for $(x_1,x_2,x_3,x_4) \in X_4(\C P^1)$ with $x_j = [z_j,w_j],\, 1 \leq j \leq 4,$ the map $cr$ is given by \[ cr(x_1,x_2,x_3,x_4) = \frac{(z_1 w_3 - z_3 w_1) (z_2 w_4 - z_4 w_2)}{(z_2 w_3 - z_3 w_2)(z_1 w_4 - z_4 w_1)}.\] In the affine chart $u_0 \times u_0 \times u_0 \times u_0$ it looks like $cr(z_1,z_2,z_3,z_4) = \frac{(z_1 - z_3) (z_2  - z_4)}{(z_2 - z_3)(z_1 - z_4)}.$

The cross-ratio map allows us to write down a diffeomorphism (in fact isomorphism of quasi-projective varieties) \[c: X_n(\C P^1) \xrightarrow{\sim} X_3(\C P^1) \times X_{n-3}(\C \setminus \{0,1\}) ,\] \[(\overrightarrow{x},y_1,\ldots, y_{n-3}) \mapsto (\overrightarrow{x},cr(\overrightarrow{x},y_1),\ldots, cr(\overrightarrow{x},y_{n-3})),\] where $\overrightarrow{x}=(x_1,x_2,x_3)$ denotes a point in $X_3(\C P^1)$ and $cr(\overrightarrow{x},y) = cr(x_1,x_2,x_3,y)$ is the cross-ratio map. Later we shall see that this diffeomorphism is precisely what makes the proofs work, as it allows one to use the affine structure on $\C \supset \C \setminus \{0,1\}$.

Note that $\pi_c := pr_2 \circ c: X_n(\C P^1) \to X_{n-3}(\C \setminus \{0,1\}),$ where $pr_2: X_3(\C P^1) \times X_{n-3}(\C \setminus \{0,1\}) \to {X_{n-3}(\C \setminus \{0,1\})}$ is the projection to the second factor, is simply a coordinate description of the natural projection $X_n(\C P^1) \to \mathcal{M}_{0,n} \cong X_n(\C P^1)/\mathrm{PSL}(2,\C).$

Finally, note that $c$ induces an isomorphism $c_{\#}: P_n(\C P^1) \to \Z/2\Z \times P_{n-3}(\C \setminus \{0,1\})$ on fundamental groups (recall that $P_3(\C P^1) \cong \pi_1(\mathrm{PSL}(2,\C)) \cong \Z/2\Z$).

\subsubsection{Differential $1$-forms on configuration spaces}

Using the isomorphism $c,$ we introduce special differential $1$-forms on $X_n (\C P^1),$ $n \geq 4,$ that we consequently use as an intermediate step in our results. Denote by $u_1,\ldots, u_{n-3}$ the affine coordinates on $ \C^{n-3} \supset (\C \setminus \{0,1\})^{n-3} \supset X_{n-3}(\C \setminus \{0,1\}).$ Define for $\nu \in I,$ for an index set $I = \{(i;0)\}_{1 \leq i \leq n-3} \cup \{(i;1)\}_{1 \leq i \leq n-3} \cup \{(ij)\}_{1\leq i\neq j \leq n-3},$ the $\R$-valued differential $1$-form on $X_{n-3}(\C \setminus \{0,1\})$ by \[\theta_\nu = \frac{1}{2\pi} \mathrm{Im}(\alpha_\nu),\] 
with \begin{align*}
\alpha_{i;0} &= \frac{du_i}{u_i},\\
\alpha_{i;1} &= \frac{d(u_i - 1)}{u_i - 1},\\
\alpha_{ij} &= \frac{d(u_i - u_j)}{u_i - u_j}.
\end{align*}

Finally define  \[\til{\theta}_\nu = (\pi_c)^* \theta_\nu \in \Omega^1(X_n(\C P^1),\R),\] for each $\nu \in I.$

For a $1$-form $\theta$ on a manifold $Y$ and a smooth parametrized path $\gamma:[0,1] \to Y$ set
\[\int_\gamma |\theta| := \int_0^1 |\theta_{\gamma(t)}(\dot{\gamma}(t))| dt.\]
Clearly, for a smooth loop $\gamma$ we have $|\int_\gamma \theta | \leq \int_\gamma |\theta |$. Moreover, $\int_{\gamma} |\theta| = \int_{\gamma^{-1}} |\theta|$, where $\gamma^{-1}$ is the time-reversal of $\gamma$.


\subsection{Main results}


Our main technical result is:

\medskip

\begin{thm}\label{Theorem: average wordlength bound}
For an isotopy $\overline{\phi} = \{\phi^t\}$ in $\G,$ the average word norm of a trajectory $\lambda(x,\overline{\phi})$ is controlled by the $L^1$-length of $\overline{\phi}:$

\[W(\overline{\phi}) = \int_{X_n(\C P^1)\setminus Z} |[\lambda(x,\overline{\phi})]|_{P_n(S^2)} \, d\mu^{\otimes n}(x) \leq A \cdot l_1(\overline{\phi}) +B,\] for certain constants $A,B > 0.$
\end{thm}

\medskip

\begin{rmk}
Note that $W(\overline{\phi})$ depends only on the class $\til{\phi} = [\overline{\phi}] \in \til{\G}$ of $\overline{\phi}$ in the universal cover $\til{\G}$ of $\G.$
\end{rmk}

Theorem \ref{Theorem: average wordlength bound} has a number of consequences concerning the large-scale geometry of the $L^1$-metric on $\G.$ Firstly, as any quasimorphism on a finitely generated group is controlled by the word norm, we immediately obtain the following statement.

\medskip

\begin{cor}\label{Corollary: GG Lip}
The homogenization $\overline{\Phi}$ of each Gambaudo-Ghys quasimorphism $\Phi$ satisfies $$|\overline{\Phi}(\phi)| \leq C \cdot d_1(\phi,1).$$
\end{cor}

\medskip

By a theorem of Ishida \cite{Ishida}, the composition $Q(B_n(S^2)) \to Q(P_n(S^2)) \xrightarrow{GG} Q(\G),$ where the first arrow is the natural restriction map and the second is the Gambaudo-Ghys map, is an embedding. Hence for $n \geq 4,$ by results of Bestvina and Fujiwara \cite{BestvinaFujiwara}, $Q(\G)$ is an infinite-dimensional vector space. Thus by Corollary \ref{Corollary: GG Lip} the diameter of $\G$ with the $L^1$-distance is infinite.

\medskip

\begin{cor}\label{Corollary: diam}
The $L^1$-diameter of $\G$ is infinite.
\end{cor}

\medskip

Considering certain special examples of Gambaudo-Ghys quasimorphisms, and their calculations for certain autonomous flows, as in \cite{BrandenburskyESIHES}, we find for each integer $k \geq 1$ a $k$-tuple of homogeneous Gambaudo-Ghys quasimorphisms $\{\overline{\Phi}_i\}_{1 \leq i \leq k}$ and a $k$-tuple of autonomous Hamiltonian flows (one-parameter subgroups) $\{\{\phi_i^t\}_{t \in \R}\}_{1 \leq i \leq k}$ such that $\overline{\Phi}_i(\phi_j^t) = t \delta_{ij}.$ This implies the following stronger statement.


\medskip

\begin{cor}\label{Corollary: vector spaces}
The metric group $(\G,d_1)$ admits a bi-Lipschitz group monomorphism from $(\R^k,d)$ where $d$ is any metric on $\R^k$ induced by a vector-space norm.
\end{cor}

\medskip

Moreover, by an argument of Kim and Koberda \cite{KimKoberdaAntiTrees} (cf. Benaim-Gambaudo \cite{BenaimGambaudo} and Crisp-Wiest \cite{CrispWiest}), Theorem \ref{Theorem: average wordlength bound} implies the following statement, finishing an answer to a question of Kapovich \cite{KapovichRAAGs} in the case of $S^2$.  

\medskip

\begin{cor}\label{Corollary: Kapovich}
The metric group $(\G,d_1)$ admits a quasi-isometric group embedding from each right-angled Artin group endowed with the word metric.  
\end{cor}

\medskip

\begin{rmk}
We note that Corollary \ref{Corollary: Kapovich} implies Corollary \ref{Corollary: diam}, providing the latter with a proof that does not use quasimorphisms.
\end{rmk}

Finally, the arguments in \cite{BrandShelAut} combined with Corollary \ref{Corollary: GG Lip} imply the following.

\medskip

\begin{cor}
For each positive integer $k,$ the complement in $\G$ of the set $Aut^k$ of products of at most $k$ autonomous diffeomorphisms contains a ball of any arbitrarily large radius in the $L^1$-metric.
\end{cor}

\medskip

\begin{rmk}\label{Remark: Jensen 1 to p}
Let $p \geq 1.$ Note that since, by Jensen's (or H\"{o}lder's) inequality, \[d_1 \leq d_p,\] all the above results for $d_1$ continue to hold for $d_p.$
\end{rmk}


\subsection{Outline of the proof}

Theorem \ref{Theorem: average wordlength bound} is an immediate consequence of the following lemma and two propositions. The lemma states that for our purposes two different choices of short paths are equivalent.

\medskip

\begin{lma}\label{Lemma:W and affine segments}
Choosing as short paths the component-wise affine segments $\gamma'(x)$ in the chart $u_{0} \times \ldots \times u_0: \C^n \to \C P^n$ to the basepoint, obtain from the isotopy $\{\phi^t\}$ another family of loops $\lambda'(x,\{\phi^t\})$ for $x \in X_n(\C P^1) \setminus Z',$ for a different negligible subset $Z',$ and hence another average word norm function $W'(\overline{\phi}) = \int_{X_n(\C P^1)\setminus Z'} |[\lambda'(x,\overline{\phi})]|_{P_n(S^2)}.$  Then $|W(\{\phi\}) - W'(\{\phi \})| \leq C$ for a constant $C$ depending only on the systems of paths. 
\end{lma}

\medskip

The first proposition is a purely topological fact about the word norm of the classes of loops in the fundamental group of the configuration space.

\medskip

\begin{prop}\label{Prop: top}
Let $\lambda$ be a piecewise $C^1$ loop in $X_n(S^2)$ based at $q.$ Let $S$ be a finite generating set of $P_n(S^2).$  The word norm of the class $[\lambda] \in \pi_1(X_n(S^2),q) \cong P_n(S^2)$ with respect to $S$ satisfies 

\[ |[\lambda]|_{S} \leq A_0 \cdot \sum_{\nu \in I} \int_{\lambda} |\til{\theta}_\nu| + B_0,\] for constants $A_0,B_0 > 0$ depending only on $S$ and on $n.$ 

\end{prop}

The second lemma is purely analytical and relies on the fact that we work with area-preserving diffeomorphisms, as well as on the fact that the differential forms we consider have integrable singularities near the divisors of $(\C P^1)^n$ that we excise to obtain $X_n(\C P^1).$

\medskip

\begin{prop}\label{Prop: analysis}
There exist constants $A_1,B_1 > 0$ depending only on $n,$ such that for each $\nu \in I,$ \[\int_{X_n(\C P^1)\setminus Z} \left(\int_{\lambda(x,\overline{\phi})} |\til{\theta}_\nu|\right)  \, d\mu^{\otimes n}(x) \leq A_1 \cdot l_1(\overline{\phi}) +B_1.\]
\end{prop}

\subsection*{Acknowledgements.}

Part of this work has been carried out during Brandenbursky's stay
at IHES and CRM Montreal. He wishes to express his gratitude to both institutes.
Brandenbursky was supported by CRM-ISM fellowship and NSF grant No. 1002477.

This work was carried out during Shelukhin's stay in CRM Montreal, ICJ Lyon 1, Institut Mittag Leffler,
and IAS. He thanks these institutions for their warm hospitality. He was supported by CRM-ISM fellowship,
FP7-IDEAS-ERC Grant no. 258204 RealUman, Mittag Leffler fellowship, and NSF grant No. DMS-1128155.

We thank Leonid Polterovich and Pierry Py for useful comments on the manuscript.

We thank Mikhail Belolipetsky, Octav Cornea, Steven Hurder, Jake Solomon, and Leonid Polterovich, for the possibility to speak on a preliminary version of these results on seminars in IMPA, CRM, UIC, Hebrew University, and the University of Chicago.

\section{Proofs.}

\begin{proof}[Proof of Lemma \ref{Lemma:W and affine segments}]
We note that for any negligible subset $Z''$ of $X_n(\C P^1),$ \[W(\overline{\phi}) = \int_{X_n(\C P^1)\setminus (Z \cup Z' \cup Z'')} |[\lambda(x,\overline{\phi})]|_{P_n(\C P^1)} \, d\mu^{\otimes n}(x)\] and \[W'(\overline{\phi}) = \int_{X_n(\C P^1)\setminus (Z \cup Z' \cup Z'')} |[\lambda'(x,\overline{\phi})]|_{P_n(\C P^1)} \, d\mu^{\otimes n}(x),\] whether these integrals are finite or not (simply by the definition of the Lebesgue integral).

Hence it is sufficient to show that there exists a constant $C$ depending only on the systems of paths and a negligible subset $Z''$ of $X_n(\C P^1),$ such that for each $x \in X_n(\C P^1) \setminus (Z \cup Z' \cup Z'') = \C^n \setminus ((\C^n \cap Z) \cup (\C^n \cap Z') \cup (\C^n \cap Z''))$ we have $||[\lambda'(x,\overline{\phi})]|_{P_n(\C P^1)} - |[\lambda(x,\overline{\phi})]|_{P_n(\C P^1)}| \leq C.$

And indeed we see that $[\lambda(x,\overline{\phi})] = [\delta(x)]^{-1}[\lambda'(x,\overline{\phi})][\delta(x)],$ for $\delta(x) = \gamma(x)\# \overline{\gamma}'(x)$ and $[\delta(x)]_{P_n(\C P^1)} \leq C,$ as can be seen by direct calculation on braid diagrams in $\C.$ Indeed, as spherical geodesics map to circular arcs or affine rays under stereographic projection, and the latter happens for $x$ in a negligible subset $Z''$ of $\C^n \setminus ((\C^n \cap Z) \cup (\C^n \cap Z'))$, considering for $x \in X_n(\C P^1) \setminus (Z \cup Z' \cup Z'')$ the diagram of the geometric braid $\delta(x)$ in a generic direction $\om \in S^1$, we see that it has at most $\displaystyle 4 {n \choose 2} + {n \choose 2}$ crossings, corresponding to the $\gamma(x)$ and $\overline{\gamma}'(x)$ parts of the geometric braid. Therefore $[\delta(x)]_{B_n(\C)} \leq \displaystyle 5 {n \choose 2}.$ However, $[\delta(x)]_{P_n(\C)} \leq A \cdot [\delta(x)]_{B_n(\C)} + B$ for constants $A,B > 0$ (see Lemma \ref{Lemma - finite index q-i} below), and obviously $[\delta(x)]_{P_n(\C P^1)} \leq [\delta(x)]_{P_n(\C)}.$ This finishes the proof. \end{proof}

\begin{proof}[Proof of Proposition \ref{Prop: top}]
Let $\overline{\lambda}$ be a loop in $X_{n-3}(\C \setminus \{0,1\})$ based at $\overline{q} = \pi_c(q).$ We first claim that the word norm of $[\overline{\lambda}]$ in $P_{n-3}(\C \setminus \{0,1\}) \cong \pi_1(X_{n-3}(\C \setminus \{0,1\}))$ satisfies \begin{equation}
\label{Equation:bound in C minus two points} |[\overline{\lambda}]|_{P_{n-3}(\C \setminus \{0,1\})} \leq  A_2 \cdot \sum_{\nu \in I} \int_{\overline{\lambda}} |\theta_\nu| + B_2\end{equation} for $A_2,B_2 > 0.$ Proposition \ref{Prop: top} follows immediately from this statement by setting $\overline{\lambda} = \pi_c \circ \lambda,$ since the map $P_n(\C P^1) \to P_{n-3}(\C \setminus \{0,1\})$ induced by $\pi_c$ is a quasi-isometry (note that it is identified with the projection $\Z/2\Z \times P_{n-3}(\C \setminus \{0,1\}) \to P_{n-3}(\C \setminus \{0,1\})$ to the second factor, under the isomorphism given by $c$).

We require the following two lemmas from geometric group theory.

\medskip

\begin{lma}\label{Lemma - adding strands q-i}
The natural map $e: P_{n-3}(\C \setminus \{0,1\}) \to P_{n-1}(\C)$ induced by adding constant strands at the punctures $\{0,1\}$ is a quasi-isometric embedding of groups.
\end{lma}

\medskip

\begin{lma}\label{Lemma - finite index q-i}
The inclusion $P_{n-1}(\C) \to B_{n-1}(\C)$ is a quasi-isometric embedding of groups.
\end{lma}

\medskip

Lemma \ref{Lemma - finite index q-i} is a consequence of a general fact on co-compact group actions (\cite[Corollary 24]{TopicsGGT}), as $P_{n-1}$ is a subgroup of finite index in $B_{n-1}.$ Lemma \ref{Lemma - adding strands q-i} is rather special to our case, and hence we provide a proof.

\medskip

\begin{proof}[Proof of Lemma \ref{Lemma - adding strands q-i}]
The map $e: P_{n-3}(\C \setminus \{0,1\}) \to P_{n-1}(\C)$ fits into the following exact sequence \cite{TuraevKassel}:

\[1 \to P_{n-3}(\C \setminus \{0,1\}) \to P_{n-1}(\C) \to P_2(\C) \to 1.\]

Note that $P_2(\C) \cong \Z.$ Moreover the generator $z$ of the center $Z(P_{n-1}(\C))$ of $P_{n-1}(\C)$ maps to a generator $1 \in \Z$ of $P_2(\C).$ Hence mapping $1$ to $z$ determines a section for $P_{n-1}(\C) \to P_2(\C)$ that yields an isomorphism between the above exact sequence and \[1 \to P_{n-3}(\C \setminus \{0,1\}) \to  P_{n-3}(\C \setminus \{0,1\}) \times \Z \to \Z \to 1,\] the first map taking the form $x \mapsto (x,0),$ and the second map being the projection to the second coordinate. The statement follows. \end{proof}


Consider the geometric braid $\overline{\lambda}$ and add two constant strands at $0$ and $1.$ Call the new geometric braid $\overline{\lambda}' \in \pi_1(X_{n-1}(\C), \overline{q} \cup \{0,1\}).$ We show that for any geometric braid $\beta \in \pi_1(X_{n-1}(\C), \overline{q} \cup \{0,1\}),$ its word norm in $B_{n-1}(\C)$ satisfies:

\begin {equation} \label{Equation:bound in C }|[\beta]|_{B_{n-1}(\C)} \leq A_3 \cdot \displaystyle\sum_{1 \leq i < j \leq n-1} \int_{\beta} |\theta'_{ij}| + B_3 \end{equation} for some $A_3, B_3 > 0$ and $\theta'_{ij} = \frac{1}{2\pi} \mathrm{Im}(\frac{d(u_i - u_j)}{u_i - u_j})$ for $1 \leq i,j \leq n-1.$ Note that the forms $\{ \theta'_{\nu'} \}_{\nu' \in I'}$ with $I' = \{(ij)\, |\; 1 \leq i < j \leq n,\; (i,j) \neq (n-2,n-1)\}$ pull back to $\{ \theta_\nu \}_{\nu \in I}$ under the natural embedding $X_{n-3}(\C \setminus \{0,1\}) \to X_{n-1}(\C)$ given by $(u_1,\ldots,u_{n-3}) \mapsto (u_1, \ldots, u_{n-3},0,1),$ and the form $\theta_{n-2,n-1}$ pulls back to the zero form. Hence by Lemmas \ref{Lemma - adding strands q-i} and \ref{Lemma - finite index q-i} the estimate \eqref{Equation:bound in C } implies the estimate \eqref{Equation:bound in C minus two points}.

Note that $\theta'_{ij} = p_{ij}^* \theta$, with $\theta = \frac{1}{2\pi} \mathrm{Im}(\frac{d(u-v)}{u-v}).$ Hence we have $\int_{\beta} |\theta'_{ij}| = \int_{\beta_{ij}} |\theta|$, where
$\beta_{ij} = p_{ij}\circ \beta$, and moreover the following equality holds by the co-area formula (see \cite[Theorem 5.1.12]{MR2427002} or \cite{GambaudoLagrange}).


For almost all $\omega \in S^1$, the quantity $$n_{ij}(\omega) = \#\{ t \in [0,1) | \frac{p_i\circ \beta (t) - p_j \circ \beta (t)}{|p_i\circ \beta(t) - p_j \circ \beta(t)|} = \om \in S^1\}$$
is finite, and defines an $L^1$-function with norm \[\int_{S^1} n_{ij}(\omega) dm(\om) = \int_{\beta} |\theta'_{ij}|\]
for $m$ the Haar (Lebesgue) measure on $S^1$. Note that (cf. \cite{BrandenburskyLpMetrics, BrandenburskyKedra1}) $n_{ij}(\omega)$ is the number of times that the $i$-th strand overcrosses the $j$-th strand in the diagram of the braid $\beta$ obtained by projection in the direction $\om$. 

We claim that there exists a constant $C$ (which depends only on $n$) and $\om \in S^1,$ such that for all $1 \leq i, j \leq n-1$ \[n_{ij}(\om) \leq C \int_{\beta} |\theta'_{ij}|\] and all crossings are transverse. Indeed, if $n_{ij}$ does not have vanishing $L^1$-norm, we estimate by Markov's (or Chebyshev's) inequality (cf. \cite[Section 29.2, Theorem 5]{KolmogFom}) \[m(\{\om \in S^1|\; n_{ij}(\om) \geq C \int_{\beta} |\theta'_{ij}|\}) \leq \frac{1}{C}.\] Hence any $C > (n-1)(n-2)/2$ would be sufficient to ensure that the intersection \[\bigcap_{1 \leq i<j\leq n-1} \{\om \in S^1|\; n_{ij}(\om) \leq C \int_{\beta} |\theta'_{ij}|\}\] has positive measure and hence is non-empty. Moreover, clearly the set of all $\om$ for which all crossings are transverse has full measure.

Hence, from the $\om$-projection diagram of the braid $\beta$ we get a presentation of $\beta$ as a word in the full braid group $B_n(\C)$, generated by say the half-twists, that has exactly one generator for each overcrossing. Hence \[|[\beta]|_{B_{n-1}(\C)} \leq \sum_{i \neq j} n_{ij}(\omega) \leq 2C \sum_{i<j} \int_{\beta}  |\theta_{ij}|.\] This finishes the proof. \end{proof}


\medskip

\begin{proof}[Proof of Proposition \ref{Prop: analysis}]
We proceed to prove the analytic estimate on averages. First we show that for each $\nu \in I$ the integral of $|\til{\theta}_\nu|$ on each of the short paths is universally bounded. 

\medskip

\begin{lma}\label{Lemma:short paths}
For each $\nu \in I$ and $x \in X_n(\C P^1) \setminus Z$ we have $\int_{\gamma'(x)} |\til{\theta}_\nu| \leq C.$ 
\end{lma}
\medskip

\begin{proof}[Proof of Lemma \ref{Lemma:short paths}]
Recall that by definition of $\gamma'(x),$ we work in the chart $\C^n.$ Since $\gamma'(x)$ is a component-wise affine segment, any linear function $h = z_i - z_j$  composed with $\gamma'(x)$ is an affine segment in $\C 
\setminus \{0\}.$ Therefore $\int_{\gamma'(x)} |\frac{dh}{h}| \leq \pi.$ By the definition of $\til{\theta}_\nu$ (see \eqref{Lemma:additivity of log der} below) we obtain that $\int_{\gamma'(x)} |\til{\theta}_\nu| \leq \frac{1}{2\pi}\cdot 6 \cdot \pi = 3$ for all $\nu \in I.$  \end{proof}

By Lemma \ref{Lemma:short paths} is sufficient to give a bound on \[ \int_{X_n(\C P^1)\setminus Z} \left(\int_{\overline{\phi} \cdot x} |\til{\theta}_\nu|\right)  \, d\mu^{\otimes n}(x),\] which by preservation of area and continuity can be rewritten as \[ \int_0^1 \left(\int_{X_n(\C P^1)} |\til{\theta}_\nu(X_t^{\oplus n})|(x)  \, d\mu^{\otimes n}(x)\right) dt.\] We note that this is the only place in the proof that uses area-preservation.

Now note that under the standard stereographic projection the lower hemisphere in $S^2$ is identified with the standard unit disk $\D=\{|z| \leq 1\}$ in $\C$. This embeds as \[H_0=\{[z,1]| |z| \leq 1\}\] in $\C P^1$ under the standard affine chart $u_0: \C \;\tilde{\rightarrow}\;U_0$ containing the point $[0,1]$. Similarly, the upper hemisphere is identified with the subset \[H_\infty=\{[1,w]| |w| \leq 1\}\] of the image of the affine chart $u_\infty: \C \;\tilde{\rightarrow}\; U_\infty$ in $\C P^1$. Moreover, $\C P^1$ is the measure-disjoint union of $H_0$ and $H_\infty$.

Write $(\C P^1)^n = \displaystyle{\bigcup_{\eps \in \{0,\infty\}^n}} H_\eps$ as a measure-disjoint union of products $H_\eps = H_{\eps_1} \times \dots \times H_{\eps_n}$ of hemispheres. Let $e_\eps$ denote the isomorphism $e_\eps=e_{\eps_1} \times \ldots \times e_{\eps_n}: \D^n = \D \times \ldots \times \D \to H_\eps,$ for $e_{\eps_j} = u_{\eps_j}|_{\D}:\D \to H_{\eps_j},\; 1\ \leq j \leq n$ given by the embeddings above. Denote $H'_\eps = H_\eps \cap X_n(\C P^1)$ and $\Delta_\epsilon = e_\eps^{-1}(H_\eps \setminus H'_\eps).$ Put $\overline{e}_\eps= e_\eps|_{H'_\eps}:\D^n\setminus \Delta_\eps \to H'_\eps.$ Then \[\int_{X_n(\C P^1)} |\til{\theta}_\nu(X_t^{\oplus n})|(x)  \, d\mu^{\otimes n}(x) = \sum_{\eps \in \{0,\infty\}^n} \int_{H'_\eps} |\til{\theta}_\nu(X_t^{\oplus n})|(x)  \, d\mu^{\otimes n}(x).\] Write  \[\int_{H'_\eps} |\til{\theta}_\nu(X_t^{\oplus n})|(x)  \, d\mu^{\otimes n}(x) = \int_{\D^n \setminus \Delta_\eps} |(\overline{e}_\eps)^*\til{\theta}_\nu(e_\eps^* X_t^{\oplus n})|(x)  \, d e_\eps^* (\mu^{\otimes n})\]

Note that $e_\eps^* X_t^{\oplus n} = e_{\eps_1}^*X_t \oplus \ldots \oplus e_{\eps_n}^*X_t,$ and that for each $\nu \in I,$ $(\overline{e}_\eps)^* \til{\theta}_\nu = \frac{1}{2\pi} \mathrm{Im}(\frac{df_{\epsilon,\nu}}{f_{\eps,\nu}})$ with $f_{\eps,\nu} = \frac{l_1 l_2}{l_3 l_4}$ or $f_{\eps,\nu} = \frac{l_1 l_2 l_3}{l_4 l_5 l_6}$ with each $l_k$ of the form $a-b$ or $ab -1$ where $a$ and $b$ are natural coordinates on two of the factors in the product $\D^n.$ More precisely, if $(a_1,\ldots,a_n)$ are natural coordinates on $\D^n,$ then $\Delta_\eps = \bigcup_{1 \leq i < j \leq n} \{h_{ij} = 0\},$ with $h_{ij} = a_i - a_j$ if $\eps_i = \eps_j$ and $h_{ij} = a_i a_j -1$ if $\eps_i \neq \eps_j,$ and for each $1 \leq k \leq 6,$ $l_k = h_{ij}$ for some $1\leq i<j \leq n.$
Indeed, this follows immediately from the identities \[cr(x_1,x_2,x_3,x_4) - 1 = -cr(x_2,x_1,x_3,x_4) = \frac{(z_1 w_2 - z_2 w_1) (z_3 w_4 - z_4 w_3)}{(z_2 w_3 - z_3 w_2)(z_1 w_4 - z_4 w_1)}\] and \[cr(x_1,x_2,x_3,x_4) - cr(x_1,x_2,x_3,x_5) = \frac{(z_1 w_3 - z_3 w_1) (z_1 w_2 - z_2 w_1) (z_5 w_4 - z_4 w_5)}{(z_2 w_3 - z_3 w_2)(z_1 w_4 - z_4 w_1)(z_1 w_5 - z_5 w_1)}\] for $x_j = [z_j,w_j]$ in homogeneous coordinates on $\C P^1$ for all $1 \leq j \leq 5.$

We record the formula \begin{equation}\label{Lemma:additivity of log der}
(\overline{e}_\eps)^* \til{\theta}_\nu = \frac{1}{2\pi} \mathrm{Im}(\frac{dl_1}{l_1} + \frac{dl_2}{l_2} + \frac{dl_3}{l_3} - \frac{dl_4}{l_4} - \frac{dl_5}{l_5} - \frac{dl_6}{l_6}),  
\end{equation}
which follows immediately from the above discussion. From \eqref{Lemma:additivity of log der} it follows that it is sufficient to estimate \[\int_{\D \times \D \setminus \{h_{ij} = 0\}} \left | \frac{dh_{ij}}{h_{ij}} \right | (e_{\eps_i}^* X_t \oplus e_{\eps_j}^* X_t) \; (e_{\eps_i} \times e_{\eps_j})^* \, d \mu^{\otimes 2}\] for each $1 \leq i < j \leq n.$
 
The pullback $|\cdot|_{Sph} = e_{\eps}^* |\cdot|_{S^2}$, $\eps \in \{0,\infty\}$ of the metric on the sphere in either of the coordinate charts is equal \[(|\cdot|_{Sph})_{\zeta} = (1+|\zeta|^2)^{-1} \, |\cdot|_{Eucl}.\] Abbreviating $|\cdot|_{Eucl} = |\cdot|$  we therefore have (for all $\zeta \in \D$) \[\frac{1}{2} |\cdot| \leq |\cdot|_{Sph} \leq |\cdot|.\] Hence in order to obtain an estimate via $\int_{\D} |(e_{\eps_k}^* X_t)|_{Sph} \; d\mu$ for $1 \leq k \leq n,$ it is sufficient to estimate via $\int_{\D} |(e_{\eps_k}^* X_t)| \, d\mu.$



Hence it is sufficient to estimate \[\int_{\D^2 \setminus \{a-b = 0\}} \left |\frac{d(a-b)}{a-b} \right |(A_t \oplus B_t) \; d\mu^{\otimes 2}\] or \[\int_{\D^2 \setminus \{ab - 1 = 0\}} \left |\frac{d(ab-1)}{ab-1} \right |(A_t \oplus B_t) \; d\mu^{\otimes 2}.\] Here \[d\mu(\zeta) = 2 (1 + |\zeta|^2)^{-2} dm(\zeta)\] is the pullback of the spherical measure to $\D$ by the any of the maps $e_{\eps_j}$ (note that the map $\C P^1 \to \C P^1$ given by $[z,w] \mapsto [w,z]$ is an isometry of the spherical metric, and hence preserves the volume form), and $A_t, B_t$ are $e_{\eps_i}^* X_t, e_{\eps_j}^* X_t$ for appropriate $i, j.$  

Start with \[\int_{\D^2 \setminus \{a-b = 0\}} \frac{|A_t(a)-B_t(b)|}{|a-b|} d\mu(a)d\mu(b).\]

We apply the triangle inequality $|A_t(a)-B_t(b)| \leq |A_t(a)|+ |B_t(b)|$, and estimate the two resulting terms separately. Since they are estimated analogously, we show the estimate for the first term only. We have
\[\int_\D  \int_\D \frac{|A_t(a)|}{|a-b|} \; d{\mu}(b) d{\mu}(a)= \int_\D |A_t(a)| \left ( \int_\D \frac{1}{|a-b|} d{\mu}(b) \right )  d{\mu}(a) \leq C \cdot  \int_\D |A_t(a)|  \, d{\mu}(a),  \]


 since 

$$  \int_\D \frac{1}{|a-b|} d{\mu}(b) \leq  2 \int_\D \frac{1}{|a-b|} d m(b) \leq  8\pi = C.$$


We continue with
\[\int_{\D^2 \setminus \{ab - 1 = 0\}} \frac{|a B_t(b) - A_t(a) b|}{|ab - 1|} \; d\mu(a)d\mu(b).\]

Using the triangle inequality in the numerator, we estimate the two terms separately. Consider for example the first term. The estimate proceeds analogously to the previous case, the only difference being following calculation. Denoting $a^*=\frac{1}{a}$, we compute
\[\int_\D \frac{|a|}{|ab-1|} d\mu(a) = \int_{\{|a^*| \geq 1\}} \frac{1}{|b-a^*|} d\mu(a^*)
\leq  C.\]

Indeed, write the last integral as the sum of the integrals over the measure-disjoint subsets $\{|a^*| \leq 2\}$ and $\{|a^*| \geq 2 \}$ of $\C.$ Then we estimate

\[\int_{\{|a^*| \leq 2\}} \frac{1}{|b-a^*|}\frac{1}{(1+|a^*|^2)^2} dm(z) \leq \int_{\{|a^*| \leq 2\}} \frac{1}{|b-a^*|} dm(a^*) \leq \]

\[ \leq \int_{\{|b-a^*| \leq 3\}} \frac{1}{|b-a^*|} dm(b-a^*) = 6\cdot \pi,\]

and recalling that $|b| \leq 1,$

\[\int_{\{|a^*| \geq 2\}} \frac{1}{|a^*-b|}\frac{1}{(1+|a^*|^2)^2} dm(a^*) \leq \int_{|a^*| \geq 2} \frac{1}{|a^*|-|b|}\frac{1}{(1+|a^*|^2)^2} dm(a^*) \leq \]

\[ \leq \int_{\{|a^*| \geq 2\}} \frac{2}{|a^*|}\frac{1}{(1+|a^*|^2)^2} dm(a^*) = C_{1} < \infty. \]
This gives us an estimate as required with $C = 12\cdot \pi + 2 C_1$. \end{proof}


\section{Examples of quasimorphisms and bi-Lipschitz embeddings of vector spaces}\label{Section: examples}

For $\alpha\in P_n = P_n(\C)$ we denote by $\widehat{\alpha}$ the $n$-component link which is a closure of $\alpha$, see Figure \ref{fig:braid-closure1}.
\begin{figure}[htb]
\centerline{\includegraphics[height=1.5in]{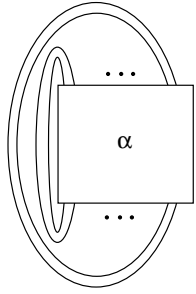}}
\caption{\label{fig:braid-closure1} Closure $\widehat{\alpha}$ of a braid $\alpha$.}
\end{figure}

Let $\sign_n\colon P_n\to \mathbb{Z}$ be a map such that $\sign_n(\alpha)=\sign(\widehat{\alpha})$, where $\sign$ is a signature invariant of links in $\mathbb{R}^3$. Gambaudo-Ghys \cite{GambaudoGhysCommutators} showed that $\sign_n$ defines a quasimorphism on $P_n$ (see \cite{BrandenburskyKnots} for a different proof). We denote by $\overline{\sign}_n\colon P_n\to \mathbb{R}$ the induced homogeneous quasimorphism. Recall that the center of $P_n$ is isomorphic to $\mathbb{Z}$. Let $\Delta_n$ be a generator of the center of $P_n$. It is a well known fact that $P_n(S^2)$ is isomorphic to the quotient of $P_{n-1}$ by the cyclic group $\langle\Delta_{n-1}^2\rangle$, see \cite{Bir}. Let $\lk_n\colon P_n\to \mathbb{Z}$ be a restriction to $P_n$ of a canonical homomorphism from $B_n=B_n(\C)$ to $\mathbb{Z}$ which takes value $1$ on each Artin generator of $B_n$. Let $s_{n-1}\colon P_{n-1}\to \mathbb{R}$ be a homogeneous quasimorphism defined by
$$s_{n-1}(\alpha):=\overline{\sign}_{n-1}(\alpha)-\frac{\overline{\sign}_{n-1}(\Delta_{n-1})}{\lk_{n-1}(\Delta_{n-1})}\lk_{n-1}(\alpha).$$
Since $s_{n-1}(\Delta_{n-1})=0$, the homogeneous quasimorphism $s_{n-1}$ descends to a homogeneous quasimorphism
$\overline{s}_n\colon P_n(S^2)\to \mathbb{R}$. Note that $\overline{s}_2$ and $\overline{s}_3$ are trivial because $P_2(S^2)$ and $P_3(S^2)$ are finite groups.

For each $n\geq 4$ let
$$\overline{\Sign}_n\colon \Diff_0(S^2,\sigma)\to \mathbb{R}$$
be the induced homogeneous quasimorphism.
In \cite[Section 5.3]{GambaudoGhysCommutators} Gambaudo-Ghys evaluated quasimorphisms $\overline{\Sign}_{2n}$ on a family of diffeomorphisms $$f_\o\colon S^2\to S^2,$$
such that $f_\o(\infty)=\infty$ and $f_\o(x)=e^{2i\pi\o(|x|)}x$, here $S^2$ is identified with $\mathbb{C\cup\{\infty\}}$, and
$\o\colon \mathbb{R}_+\to\mathbb{R}$ is a function which is constant in a neighborhood of $0$ and outside
some compact set. Let $a(r)$ be the spherical area (with the normalization $\vol(\C) = 1$) of the disc in $\mathbb{C}$ with radius $r$ centered at $0$. Set $u=1-2a(r)$ and let $\widetilde{\o}(u)=\o(r)$. In \cite[Lemma 5.3]{GambaudoGhysCommutators} Gambaudo-Ghys showed that for each $n\geq 2$
\begin{equation}\label{eq:GG-sign-comp}
\overline{\Sign}_{2n}(f_\o)=\frac{n}{2}\int\limits_{-1}^1(u^{2n-1}-u)\widetilde{\o}(u)du.
\end{equation}


\begin{proof}[Proof of Corollary \ref{Corollary: vector spaces}]
Let $H_\o\colon S^2\to \mathbb{R}$ be a smooth function supported away from the $\{\infty\}$ point and $f_{t,\o}$ be a Hamiltonian flow generated by $H_\o$, such that $f_{1,\o}=f_\o$. Since $f_{t,\o}$ is an autonomous flow, by \eqref{eq:GG-sign-comp} we have
$$
\overline{\Sign}_{2n}(f_{t,\o})=t\frac{n}{2}\int\limits_{-1}^1(u^{2n-1}-u)\widetilde{\o}(u)du.
$$

Let $d\in \mathbb{N}$. It follows from \eqref{eq:GG-sign-comp} that it is straight forward to construct a family of functions $\o_i\colon \mathbb{R}_+\to\mathbb{R}$ and $\{H_{\o_i}\}_{i=1}^d$ supported away from the $\{\infty\}$ point such that
\begin{itemize}
\item
Each Hamiltonian flow $f_{t,\o_i}$ is generated by $H_{\o_i}$ and $f_{1,\o_i}=f_{\o_i}$.
\item
The functions $\{H_{\o_i}\}_{i=1}^d$ have disjoint support and hence the diffeomorphisms $f_{t,\o_i}$ and $f_{s,\o_j}$ commute for all $s,t\in \mathbb{R}$, $1\leq i,j\leq n$.
\item
The $(d\times d)$ matrix
$\left(
              \begin{array}{ccc}
                \overline{\Sign}_{4}(f_{1,\o_1}) & \cdots & \overline{\Sign}_{4}(f_{1,\o_d}) \\
                \vdots & \vdots & \vdots \\
                \overline{\Sign}_{2d+2}(f_{1,\o_1}) & \cdots & \overline{\Sign}_{2d+2}(f_{1,\o_d}) \\
              \end{array}
            \right)
$
is non-singular.
\end{itemize}

It follows that there exists a family $\{\overline{\Phi}_i\}_{i=1}^d$ of homogeneous quasimorphisms on $\Diff_0(S^2,\sigma)$, such that $\overline{\Phi}_i$ is a linear combination of $\overline{\Sign}_{4},\ldots,\overline{\Sign}_{2d+2}$ and
\begin{equation}\label{eq:sign-basis}
\overline{\Phi}_i(f_{t,\o_j})=\left\{
                                          \begin{array}{c}\begin{aligned}
                                            &t  &\rm{if}&\quad i=j\\
                                            &0  &\rm{if}&\quad i\neq j\\
                                            \end{aligned}
                                          \end{array}
                                        \right ..
\end{equation}

Let $I\colon\mathbb{R}^d\to \Diff_0(S^2,\sigma)$ be a map, such that
$$I(v):=f_{v_1,\o_1}\circ\ldots\circ f_{v_d,\o_d}$$
and $v=(v_1,\ldots,v_d)$. It follows from the construction of $\{f_{v_i,\o_i}\}_{i=1}^d$ that $I$ is a monomorphism. Let $A'_p:=\max\limits_i \,{l_p(\{f_{t,\o_i}\}_{0 \leq t \leq 1})}$, then
$$\|f_{v_1,\o_1}\circ\ldots\circ f_{v_d,\o_d}\|_{p} \leq A'_p \|v\|,$$
where $\|v\|=\sum\limits_{i=1}^d |v_i|,$ and $\|\cdot\|_p = d_p(\cdot,1)$ denotes the $L^p$-norm.


All diffeomorphisms $f_{v_1,\o_1},\ldots,f_{v_d,\o_d}$ pair-wise commute. Hence, for each $1\leq i\leq d$, by Corollary \ref{Corollary: GG Lip} and \eqref{eq:sign-basis} we have
$$\|f_{v_1,\o_1}\circ\ldots\circ f_{v_d,\o_d}\|_{p}\geq A_p^{-1}\left|\overline{\Phi}_i(f_{v_1,\o_1}\circ\ldots\circ f_{v_d,\o_d})\right|=
A_p^{-1}\cdot |v_i|\left|\overline{\Phi}_i(f_{1,\o_i})\right|,$$
where $A_p$ is the maximum over the Lipschitz constants (in Corollary \ref{Corollary: GG Lip}) of the functions
$$\overline{\Phi}_i\colon\Diff_0(S^2,\sigma)\to \mathbb{R}.$$
It follows that
$$\|f_{v_1,\o_1}\circ\ldots\circ f_{v_d,\o_d}\|_{p}\geq \left((d\cdot A_p)^{-1}\min_i\left|\overline{\Phi}_i(f_{1,\o_i})\right|\right) \|v\|=
\left((d\cdot A_p)^{-1}\right) \|v\|,$$
and the proof follows. \end{proof}

\appendix

\section{The case of the torus}

By \cite[Theorem 1.2]{BrandenburskyKedra2} and Remark \ref{Remark: Jensen 1 to p}, the following statement implies that the diameter of $(\Diff_0(T^2,dx\wedge dy),d_{L^p})$ is infinite for all $p \geq 1.$

\medskip

\begin{prop}\label{Proposition: torus qi}
The inclusion $(\Ham(T^2,dx\wedge dy),d_{L^1}) \hookrightarrow (\Diff_0(T^2,dx\wedge dy),d_{L^1})$ is a quasi-isometry.
\end{prop}

\medskip

\begin{proof}[Proof of Proposition \ref{Proposition: torus qi}]
We equip the torus $T^2$ with the standard flat Riemannian metric. We use the following instance of the Flux exact sequence (see \cite{BanyagaStructure,MR1698616}): 

\[1\to \Ham(T^2,dx\wedge dy) \xrightarrow{\iota} \Diff_0(T^2,dx \wedge dy) \xrightarrow{Flux} T^2 \to 1.\]

It has the property that the monomorphism $\tau: T^2 \to \Diff_0(T^2,dx \wedge dy)$ given by $\tau(a,b):(x,y) \mapsto (x+b,y-a)$ satisfies $Flux \circ \tau = \id_{T^2}.$ In particular \[\Diff_0(T^2,dx \wedge dy) = \Ham(T^2,dx\wedge dy)\cdot \tau(T^2).\] However, $d_{L^1}(\tau(a,b),\id) \leq \frac{1}{\sqrt{2}}$ for all $(a,b) \in T^2,$ as is verified in an elementary manner. In particular, $\iota:\Ham(T^2,dx\wedge dy) \to \Diff_0(T^2,dx \wedge dy)$ has coarsely dense image.

We proceed to prove that $\iota$ is a bi-Lipschitz group monomorphism. First, $\iota^* d_{L^1} \leq d_{L^1}$ is immediate by definition of the $L^1$-distance. We claim that $c\cdot d_{L^1} \leq \iota^* d_{L^1}$ for some $0 < c < 1.$ By right-invariance, it is sufficient to show that $c\cdot d_{L^1}(h,1) \leq d_{L^1}(\iota(h),1),$ for all $h \in \Ham(T^2,dx \wedge dy).$ Consider a smooth path ${[0,1] \to \Diff_0(T^2,dx \wedge dy),\; t \mapsto g_t,}$ with $g_0 = \id,$ $g_1 = \iota(h).$ Look at the path \[{[0,1] \to \Diff_0(T^2,dx \wedge dy),\; t \mapsto \tau_t = \tau\circ Flux (g_t)}.\] Notice that in fact it is a loop based at $\id \in \Diff_0(T^2,dx \wedge dy).$ We shall prove the following estimate of $L^1$-lengths. 

\medskip

\begin{clm}\label{Claim: length estimate torus}
$ l_1(\{\tau_t^{-1}\}) \leq c_0 \cdot l_1(\{g_t\})$ for some $c_0 > 0.$
\end{clm}

We defer the proof of this claim to the end of the section. Define the path \[[0,1] \to \Ham(T^2,dx \wedge dy),\; t \mapsto h_t = \iota^{-1}(\tau_t^{-1} g_t),\] with $h_0 = \id,$ $h_1 = h.$ Then, since $\tau_t$ are isometries, we see that \[l_1(\{h_t\}) \leq l_1(\{\tau_t^{-1}\}) + l_1(\{g_t\}) \leq (1+c_0) \cdot l_1(\{g_t\}),\] by Claim \ref{Claim: length estimate torus}. This finishes the proof, with $c = (1 + c_0)^{-1}.$\end{proof}


\begin{proof}[Proof of Claim \ref{Claim: length estimate torus}]
First of all, since $\tau_t$ are isometries, $l_1(\{\tau_t^{-1}\}) = l_1(\{\tau_t\}).$ Let $Y_t = a_t(x,y) \, \partial_x + b_t(x,y) \, \partial_y$ be the time-dependent symplectic vector field generating $\{g_t\}.$ Denote for $f \in \sm{T^2,\R},$ its average by $\langle f \rangle = \int_{T^2} f \, dx \wedge dy$ (our area form has total area $1$). We record that \begin{equation}\label{Equation: average estimate}
|\langle f \rangle| \leq |f|_{L^1}.
\end{equation}

It follows quickly from the definition of $Flux$ (and an explicit characterization of exact $1$-forms on $T^2$) that the vector field $Z_t = \langle a_t \rangle \, \partial_x + \langle b_t \rangle \, \partial_y$ generates $\tau_t.$ Hence by \eqref{Equation: average estimate} we have \[|Z_t|_{L^1} \leq \sqrt{2}\cdot |Y_t|_{L^1}\] for each $0 \leq t \leq 1.$ Hence $ l_1(\{\tau_t\}) \leq \sqrt{2} \cdot l_1(\{g_t\}),$ finishing the proof with $c_0 = \sqrt{2}.$ \end{proof}

\bibliographystyle{amsplain}
\bibliography{LpMetricTwoSphereRefs}

\vspace{3mm}

\textsc{Michael Brandenbusrky, Department of Mathematics, Ben Gurion University, Be’er Sheva 84105, Israel}\\
\emph{E-mail address:} \verb"brandens@math.bgu.ac.il"

\vspace{3mm}

\textsc{Egor Shelukhin, Institute for Advanced Study, Einstein Drive, Princeton NJ 08540, USA}\\
\emph{E-mail address:} \verb"egorshel@ias.edu"

\end{document}